\documentclass{article}

\usepackage{latexsym,amssymb}
\usepackage{amsmath}
\usepackage[mathscr]{eucal}
\usepackage{epic,eepic}

\usepackage{graphicx}

\usepackage{color}
\usepackage{enumerate}
\usepackage{comment}
\usepackage[
 anchorcolor=blue,%
 bookmarks=true,%
 bookmarksnumbered=true,%
 colorlinks=true,%
 citecolor=cyan,%
 linkcolor=blue%
]{hyperref}

\definecolor{refkey}{gray}{0.5}
\definecolor{labelkey}{gray}{0.2}

\usepackage{amsthm}
\newtheorem{theorem}{Theorem}[section]

\newtheorem{lemma}[theorem]{Lemma}
\newtheorem{corollary}[theorem]{Corollary}

\theoremstyle{definition}

\theoremstyle{remark}
\newtheorem{remark}[theorem]{Remark}

\theoremstyle{conjecture}
\newtheorem{conj}[theorem]{Conjecture}

\makeatletter

\@addtoreset{equation}{section}
\makeatother

\newcommand{\N}{\mathbb{N}}

\newcommand{\C}{\mathcal{C}}

\newcommand{\R}{\mathbb{R}}
\newcommand{\Sph}{\mathbb{S}}

\newcommand{\g}{\mathrm{g}_{can}}

\newcommand{\Ric}{\mathrm{Ric}}
\newcommand{\diam}{\mathrm{diam}}

\newcommand{\vol}{\mathrm{vol}}

\usepackage{cite}
\usepackage[utf8]{inputenc}
\usepackage{dsfont}
\usepackage[top=2.5cm,bottom=2.5cm,right=2.5cm,left=2.5cm]{geometry}

\begin{document}

\title{A Generalization of Caffarelli’s Contraction Theorem to Nearly Spherical Manifolds}

\author{Yuxin Ge\thanks{IMT, Université de Toulouse, 118, route de
Narbonne, 31062 Toulouse Cedex, France ({\sf yge@math.univ-toulouse.fr})}
\and
Jordan Serres\thanks{LPSM, Sorbonne Université, 4 place Jussieu 75005 Paris, France
({\sf serres@lpsm.paris})}}

\maketitle

\begin{abstract}
We show that every nearly spherical manifold can be realized as the volume-preserving image of a round sphere, via the Brenier–McCann optimal transport map. This theorem extends Caffarelli’s contraction theorem to nearly spherical manifolds and yields, as a corollary, a proof of a perturbative form of Milman’s conjecture. The proof is based on a novel stability result for optimal transport maps on the sphere.
\end{abstract}

\section{Introduction and main results}

Caffarelli's contraction theorem \cite{caffarelli1} asserts that if $\mu = e^{-V(x)}dx$ is a probability measure on $\R^n$ satisfying $\nabla^2 V \geq I_n$, then the Brenier optimal transport map that pushes forward the Gaussian measure $(2\pi)^{-n/2} e^{-|x|^2/2} dx$ onto $\mu$ is a contraction. The original proof relies on a PDE-based analysis of the Monge–Ampère equation solved by the Brenier map, although a more recent approach via entropic regularization techniques has also been established \cite{caffarellientropy, chewi2023entropic}.
The uniform log-concavity assumption 
$$\nabla^2 V \geq I_n$$ 
falls within the broader framework of $CD(1,\infty)$ spaces, which, in a synthetic sense, correspond to metric measure spaces whose Ricci curvature satisfies $\Ric \geq 1$ without any restriction on the dimension. This naturally raises the question of whether Caffarelli’s contraction theorem remains valid in the setting of $CD(\rho,n)$ spaces, which synthetically correspond to spaces having a Ricci curvature bounded below by $\rho>0$ and a finite dimensional upper bound $n>1$.

Since the model space saturating the $CD(\rho,\infty)$ condition is, for all $n \geq 1$, the $n$-dimensional Gaussian with covariance matrix $\rho \cdot I_d$, Caffarelli’s theorem can be interpreted as stating that the optimal transport map sending the standard $n$-dimensional Gaussian, which serves as the model density for the $CD(1,\infty)$ condition, onto a smooth $CD(1,\infty)$ density in $\R^n$ is a contraction.

From this viewpoint, a natural question arises: does the optimal transport map pushing forward the uniform measure on the sphere of dimension $n$ and radius $(n-1)/\rho$, which is the model space achieving equality in the $CD(\rho,n)$ condition, onto a smooth $CD(\rho,n)$ density, remain a contraction? To make the optimal transport problem meaningful, it is reasonable to require that this density be supported on the same sphere as the model one. However, several well-known obstructions occur in this setting. For instance, any contractive function $f : (\Sph^n, g_{can}) \to (\Sph^n, g_{can})$ is either supported within the interior of a hemisphere or is an isometry; see, for example, \cite[Proposition 1.41]{brudnyi2}. One may attempt to replace the uniform distribution on the entire sphere by the uniform distribution on a hemisphere \cite[Section 4]{BeckJerison}, but Caffarelli’s theorem still fails to extend in this case; see \cite{someobstructions}.\\

Parallel to the research avenue aiming to extend Caffarelli’s contraction theorem to the setting of $CD(\rho,n)$ spaces lies a more classical geometric question: what are the consequences for a space whose curvature exceeds that of the sphere?
Assume that $(M^n, g)$ is an $n$-dimensional Riemannian manifold satisfying
\begin{equation}\label{eq:riccicompar}
\Ric_g \geq (n-1) g.
\end{equation}
In this case, the geometry of $(M^n,g)$ can be compared with that of the standard sphere $(\mathbb{S}^n, \g)$ of radius $1$, which realizes equality in \eqref{eq:riccicompar} and thus serves as the canonical model space.  
A number of fundamental comparison results follow from this curvature bound, including:
\begin{itemize}
\item (Myers \cite{Meyer}) The diameter is bounded by that of the unit sphere: $\diam(M,g) \leq \diam(\mathbb{S}^n,\g) = \pi$.
\item (Bishop–Gromov \cite{bishop1964relation}) For every $\rho > 0$, $x \in M^n$, and $x' \in \mathbb{S}^n$,
\[
\vol_g\big(B(x,\rho)\big) \geq r\, \vol_{\g}\big(B(x',\rho)\big),
\]
where the volume ratio is given by $r = \vol_g(M^n) / \vol_{\g}(\mathbb{S}^n)$.
\item(Lichnerowicz \cite{Lichnerowicz1958GomtrieDG}) The first non-zero eigenvalue of the Laplace–Beltrami operator satisfies $\lambda_1(M^n,g) \geq \lambda_1(\mathbb{S}^n,\g) = n$.
\end{itemize}
Further consequences include sharp isoperimetric inequalities à la Lévy–Gromov, Gaussian-type heat kernel bounds, and logarithmic Sobolev inequalities.

Milman \cite{Milman2018} observed that these classical inequalities can be unified under a single geometric principle: the existence of a measure-preserving contraction between the model sphere and the given manifold. More precisely, if there exists a $1$-Lipschitz map
\[
T : (\mathbb{S}^n, \g, \vol_{\g}) \to (M^n, g, \vol_g)
\]
such that $T_\#(r\, \vol_{\g}) = \vol_g$, where
\[
r = \frac{\vol_g(M^n)}{\vol_{\g}(\mathbb{S}^n)},
\]
then Myers’ diameter estimate follows immediately from the contractivity of $T$, while the Lichnerowicz bound can be deduced from the variational characterization of $\lambda_1$. 

Motivated by this observation, Milman proposed the following conjecture, restricted to manifolds diffeomorphic to the sphere to exclude obvious topological obstructions to the existence of a Lipschitz contraction.
\begin{conj}{\cite[Conjecture 4]{Milman2018}}\label{conj:milman}
Let $(\mathbb{S}^n, g, \vol_g)$ be a Riemannian manifold with $\Ric_g \geq (n-1) g$. Then there exists a map
\[
T : (\mathbb{S}^n, \g, \vol_{\g}) \to (\mathbb{S}^n, g, \vol_g)
\]
from the unit round sphere $(\mathbb{S}^n, \g)$ that pushes forward $\vol_{\g}$ onto $\vol_g$ up to a constant factor, and which is contractive with respect to the corresponding Riemannian metrics.
\end{conj}

It is worth emphasizing that Milman’s conjecture merely requires the existence of \emph{some} contractive transport map, without imposing that it coincide with the optimal transport map. In this sense, the conjecture can be viewed as a weakened variant of the problem of extending Caffarelli’s contraction theorem to the framework of $CD(\rho,n)$ spaces for general values of $\rho$ and $n$.\\

Very recently, a perturbative version of Milman’s conjecture was established in dimension $2$ \cite{serres25}, meaning that the conjecture was verified for nearly spherical surfaces. The contractive transport map constructed in that work was the Kim–Milman map induced by the Ricci flow. To the best of the authors’ knowledge, this represents the first positive evidence toward a general proof of Milman’s conjecture. In the present work, we extend this result to arbitrary dimensions by using the Brenier–McCann optimal transport map in place of the Kim–Milman map. In particular, our main result, stated below, contributes to the ongoing development of extensions of Caffarelli’s contraction theorem to the setting of $CD$-spaces.

\begin{theorem}\label{thm:main}
Let $\alpha\in(0,1)$, let $\Sph^n$ be the standard differentiable $n$-sphere, let $g_\rho$ be the metric of the round sphere of radius $\rho\in(0,1)$, and let $g_{can}$ be the canonical metric of the round sphere of radius $1$. There exists $\varepsilon>0$ only depending on the dimension $n\geq 1$, the radius $\rho\in(0,1)$, and $\alpha\in (0,1)$, such that for all metric $g$ such that 
\begin{equation}\label{eq:nearlyspherical}
\left|\left| g - g_\rho \right| \right|_{\C^{0,\alpha}(\Sph^n)} \leq \varepsilon 
\end{equation}
the optimal transport map $T: \Sph^n\to\Sph^n$ for the quadratic cost $d_{g_{can}}^2$, seen as a map between the following manifolds $T : (\Sph^n,g_{can}) \to (\Sph^n,g)$, transporting $\vol_{g_{can}}$ onto $\frac{\vol_{g_{can}}(\Sph^n)}{\vol_{g}(\Sph^n)}\vol_g$, is 1-Lipschitz. 
\end{theorem}

We deduce the following corollary, which corresponds to a proof of Milman's conjecture restricted to nearly spherical manifolds.

\begin{corollary}\label{corol:Milmanperturbatif}
Let $(\Sph^n,g)$ be such that $\Ric\geq (n-1)g$ and assume that $g$ is different from the canonical metric $g_{can}$ of the round sphere of radius $1$. Then there exists $\varepsilon>0$ only depending on the dimension $n$, the volume of $(\Sph^n,g)$ and $\alpha\in (0,1)$, such that if 
$$\left|\left| g - g_\rho \right| \right|_{\C^{0,\alpha}(\Sph^n)} \leq \varepsilon $$
where $(\Sph^n,g_{\rho})$ denotes the round sphere of radius $\rho$, $$\rho = \left(\frac{\vol_g(\Sph^n)}{\vol_{g_{can}}(\Sph^n)}\right)^{1/n} < 1,$$
then it exists a contractive transport map $T : (\Sph^n,g_{can}) \to (\Sph^n,g)$, \textit{i.e.} $T$ is 1-Lipschitz and $$ T_\#\vol_{g_{can}} = \frac{\vol_{g_{can}}(\Sph^n)}{\vol_{g}(\Sph^n)}\vol_g. $$
\end{corollary}
\begin{remark}
Note that under the assumption $\Ric\geq (n-1)g$, it follows from Gromov-Bishop comparison theorem that $\vol_g(\Sph^n) \le \vol_{g_{can}}(\Sph^n)$  and the equality holds if, and only if, $g=g_{can}$ \cite[section 7.1.2]{Petersen}. And it is false that all small perturbations of the sphere of radius $1$ are contractive, measure preserving image of $(\Sph^n,g_{can})$. This is why we exclude this case in Corollary \ref{corol:Milmanperturbatif}. Also, our perturbation result does not require $\Ric\geq (n-1)g$ pointwise since the smallness is in $\C^{0,\alpha}$ topology.
\end{remark}

The proof is based on the following stability result for the optimal transport problem, which appears to be interesting by itself.

\begin{theorem}\label{thm:stability}
Let $(\Sph^n,g_{can})$ be the round sphere of radius $1$, and let us denote by $\vol_{can}$ its associated volume measure. Let $\alpha\in (0,1)$ and let $f_1 : \Sph^n\to \R_+$ be a $\C^{0,\alpha}$ function such that $\int_{\Sph^n} f_1\,d\vol_{can} = \vol_{can}(\Sph^n)$, and let us denote by
$\mu_1 = f_1 d\vol_{can}$ the induced positive measure (with same volume as $\vol_{can}$). Let $T : \Sph^n\to\Sph^n$  be the optimal transport mapping, with respect to the quadratic cost induced by the metric $g_{can}$, pushing forward $\vol_{can}$ onto $\mu_1$.  Then the following holds:for all $\alpha'\in (0,\alpha)$ and for all small $\delta>0$, 
there exists some $\varepsilon>0$ depending on the dimension $n$, $\alpha$ and $\alpha'$, such that if
$$\left|\left| f_1 - 1 \right|\right|_{\C^{0,\alpha}(\Sph^n)} \leq \varepsilon, $$ then it holds that
$$\left|\left| T_1 - Id \right|\right|_{\C^{1,\alpha'}(\Sph^n)} \leq \delta. $$
\end{theorem}

If we assume more regularity of density of positive probability, we get better stability result as follows.

\begin{theorem}\label{thm:stabilitybis}
Let $(\Sph^n,g_{can})$ be the round sphere of radius $1$, and let us denote by $\vol_{can}$ its associated volume measure. Let $\alpha\in (0,1)$ and let $f_1, f_2 : \Sph^n\to \R_+$ be two $\C^{0,\alpha}$ functions such that $\int_{\Sph^n} f_1\,\vol_{can} = \int_{\Sph^n} f_2\,d\vol_{can} = \vol_{can}(\Sph^n)$, and let us denote by
$\mu_1 = f_1 d\vol_{can}$ and $\mu_2 = f_2d\vol_{can}$ the two induced positive measures (with same volume as $\vol_{can}$). Let $T_1 : \Sph^n\to\Sph^n$ (respectively $T_2 : \Sph^n\to\Sph^n$) be the optimal transport mapping, with respect to the quadratic cost induced by the metric $g_{can}$, pushing forward $\vol_{can}$ onto $\mu_1$ (resp. $\vol_{can}$ onto $\,\mu_2$). We suppose $f_1$ is fixed. Then the following holds: for all $\alpha'\in (0,\alpha)$ and $k\ge 0$, for all $\delta>0$, there exists some $\varepsilon>0$ depending on the dimension $n$, $k$, $f_1$ $\alpha'$ and $\alpha$, such that if
$$\left|\left| f_1 - f_2 \right|\right|_{\C^{k,\alpha}(\Sph^n)} \leq \varepsilon, $$ then it holds that
$$\left|\left| T_1 - T_2 \right|\right|_{\C^{k+1,\alpha'}(\Sph^n)} \leq \delta. $$
\end{theorem} 

In what follows, Section \ref{sec:relatedliterature} provides a concise overview of the literature related to Caffarelli’s contraction theorem, the regularity theory for optimal transport maps, and their stability properties.
The subsequent two sections are devoted to the proofs of our main results.

\section{Related literature}\label{sec:relatedliterature}

A natural candidate for a map solving Milman's conjecture arises from the Kim–Milman construction \cite{kim2012generalization}, which provides a general method for building transport maps between two probability measures via a continuous flow interpolating between them. The Lipschitz (or contraction) property of the resulting map can then be investigated by analyzing how certain convexity features are preserved along the flow. Consequently, Conjecture~\ref{conj:milman} reduces to the task of identifying an appropriate flow sending the manifold $(\mathbb{S}^n, g)$ onto the round sphere $(\mathbb{S}^n, \g)$.

The two most natural candidates for such a flow are arguably the heat flow on the sphere and the Ricci flow. The heat flow approach was developed by Fathi, Mikulincer, and Shenfeld in \cite{FathiMikulincerShenfeld}, where they proved that the transport map induced between the uniform measure on the $n$-dimensional round sphere of radius $\sqrt{n-1}$ and an $L$-log-Lipschitz perturbation of it is $\exp(35(L+L^2))$-Lipschitz. The scaling by $\sqrt{n-1}$ ensures that this Lipschitz bound is independent of the dimension. However, since $L>0$, the constant $\exp(35(L+L^2))$ is strictly greater than one, and thus the result does not capture the desired contractive behavior.
The approach by means of the Ricci flow was pursued in \cite{serres25}, where Milman’s conjecture was verified for nearly spherical surfaces, thus establishing Theorem~\ref{thm:main} in dimension $2$. To the best of the authors’ knowledge, this provided the first positive evidence toward a general proof of Milman’s conjecture. In the present work, we extend this result to arbitrary dimensions by employing the Brenier–McCann optimal transport map in place of the Kim–Milman map generated by the Ricci flow.

We also mention the recent work of Carlier, Figalli, and Santambrogio \cite{carlier}, who extended Caffarelli’s contraction theorem to the setting of $1/d$-concave densities on $\R^d$, that is, densities of the form $V^{-d}$ with $V$ convex. This result represents a significant step toward extending Caffarelli’s theorem to the framework of $CD$-spaces, since such $1/d$-concave densities satisfy the curvature–dimension condition $CD(0,n)$ for a negative parameter $n$. We refer to \cite{othanegativedim, milmannegativedim} for a detailed discussion of the meaning of this criterion. In particular, since for any $n>1$ and $\rho\in\R$, the condition $CD(\rho,n)$ implies $CD(\rho,-n)$, their theorem may also be regarded as positive evidence supporting the development of Caffarelli-type contraction results in this extended setting.\\

The regularity of optimal transport mapping plays an important role in Optimization theory. It has interesting connections to Geometry and Physics, and other topics.  We consider the cost function equal to the half of geodesic distance square, that is, $c(x,y)=\frac12 d(x,y)^2$. In the flat case, due to Brenier \cite{brenier87, brenier91}, there exists a unique optimal transport mapping $T$ pushing forwards $f_1\, dx$ onto $f_2\, dx$. Moreover, we can find a convex potential function $u : \R^n\to \R$ such that $T=\nabla u$, and $u$ satisfies a Monge-Ampère type equation
$$
\det (\nabla^2 u)=\frac{f_1(x)}{f_2(\nabla u(x))}.
$$
On the flat tori, Cordero-Erausquin \cite{Dario99} shows the $C^{2,\alpha'}$ regularity of the potential $u$ provided that the positive density functions $f_1,f_2$ are $C^{2,\alpha}$. On Riemannian manifolds, McCann \cite{McCann01} generalizes the results due to Brenier: we can find a $c$-convex potential function $u : \mathcal{M}\to\R$ such that the unique optimal transport mapping $T : \mathcal{M} \to \mathcal{M}$ pushing forwards $f_1\, d\vol$ onto $f_2 \,d\vol$ is equal to 
\[
T(x)=\exp_x (\nabla u(x)),
\]
where $\exp_x$ denotes the Riemannian exponential map, and moreover $u$ 
satisfies a Monge-Ampère type equation
\begin{equation}\label{AMG}
\det (\nabla^2 u+\nabla^2 c(x, T(x)))=\frac{f_1(x)}{f_2(T(x))}\left|\det \nabla_x\nabla_y\, c(x, T(x))\right|
\end{equation}
In \cite{MaTrudingerWang05}, Ma, Trudinger and Wang introduce a new tensor, now called the MTW tensor after them, and show the interior $C^3$ regularity of the potential function $u$ under the condition that the MTW tensor is positive and that the probability densities $f_1, f_2$ are $C^2$. For the conveniences of the readers, we recall the definition of MTW when the cost function $c(x,y)=\frac12 d(x,y)^2$. Given two points $x,y$, we assume that $y$ is not in the cut-locus of the point $x$, that is, there is a unique minimizing geodesic between $x$ and $y$. Let $v= -\nabla_x c(x,y)$ be the tangent vector in $T_x M$. This means $\exp_x(v)=y$ where $\exp_x$ is the exponential map at the point $x$. Given two unit orthogonal tangent vectors $\xi,\nu\in T_xM$, the MTW tensor is defined (see \cite{MaTrudingerWang05,delanoe2015smoothness}) as follows
\begin{equation}
\label{MTW}
\begin{array}{ll}
{\cal C}{(x,y)}(\xi,\nu):=\left.\displaystyle\frac{d^2}{dt^2}\right|_{t=0}[-{\rm D}^2_{xx}c](x, {\rm exp}_{x}(v_0+t\nu))(\xi,\xi)\
\end{array}
\end{equation}
where $ D$ is the canonical flat connection in $T_{x}M$. Moreover, $M$ is said to be an $A3$ manifold or, equivalently we say that the $A3S$ condition is satisfied if there exists $\theta>0$ such that for all $x,y\in M$ such that $y$ is not in the cut-locus of $x$, and for all unit orthogonal tangent vector $\xi,\nu\in T_xM$, it holds
\begin{equation}
\label{A3S}
\begin{array}{ll}
{\cal C}{(x,y)}(\xi,\nu)\ge \theta.
\end{array}
\end{equation}
Later on, Kim and McCann \cite{KimMcCann10} give an intrinsic geometric interpretation of MTW tensor, by recasting the key positivity hypothesis \eqref{A3S}
of Ma, Trudinger and Wang as the non-negativity of
certain pseudo-Riemannian sectional curvatures of some pseudo-metric. In \cite{loeper11}, Loeper shows that this condition is satisfied on the standard spheres. The $C^{1,\alpha}$ regularity of the potential function $u$ is proved in \cite{loeper09,loeper11}, and Liu--Trudinger--Wang \cite{liu2009interior} show interior $C^{2,\alpha}$ regularity. For higher-order regularity, Delano\"e \cite{delanoe2015smoothness} proves the $C^{k+2,\alpha}$ regularity of the potential function for $k \ge 2$ (see also \cite{FigalliKimMcCann13b, ge2021regularity, KimMcCann12, liu2009interior}) when the probability densities $f_1, f_2$ are $C^{k,\alpha}$. There are two key estimates: on the one hand, the optimal transport map stays away from the cut locus; on the other hand, the $C^2$ \textit{a priori} estimate for the potential function $u$ are established by using the maximum principle. There are many extensions of the MTW tensor and regularity results for optimal transport maps, we refer the reader to \cite{delanoe2006gradient, DG10, DG11, Figalliloeper09, FigalliKimMcCann13, FigalliRifford09, FigalliRiffordVillani12, FigalliRiffordVillani12b, KimMcCann12, TrudingerWang09}.\\

Beyond the existence, uniqueness, and regularity of the optimal transport mapping $T_{\rho\to\mu}$, one may ask about its \emph{stability} with respect to variations of the marginal distributions $\rho$ and $\mu$. In the flat case, Brenier showed \cite{brenier91} that, for a fixed measure $\rho$, the following optimal transport functional
\begin{align*}
\left(\mathcal{P}_2(\R^d),\, W_2 \right) &\to L^2(\rho), \\
 \mu &\mapsto T_{\rho\to\mu},
\end{align*}
is continuous. This continuity property encodes a \emph{qualitative stability} of the optimal transport mapping. Recently, there has been growing interest in \emph{quantitative stability}, which provides effective bounds useful for numerical purposes and also leads to a better theoretical understanding of the phenomenon. The quantitative results obtained in the literature are typically of the form
\begin{equation}\label{eq:stabilityformgeneral}
\forall \mu_1,\mu_2\in \mathcal{P}_2,\quad \left\| T_{\rho\to\mu_1} - T_{\rho\to\mu_2}\right\|_{L^2(\rho)}  \leq C\, W_1(\mu_1,\mu_2)^{\alpha},
\end{equation}
for some constants $C,\alpha>0$. These quantitative bounds do not always hold; see \cite{Letrouitinstable}. However, several important results have recently been obtained under regularity assumptions on $\rho$; see Mérigot, Delalande, and Chazal \cite{merigotDelalandeChazal}, Gallouët, Mérigot, and Thibert \cite{GallouetMerigotThibert}, Delalande and Mérigot \cite{DelalandeMerigo23}, and Letrouit and Mérigot \cite{LetrouitMerigot24}. Quantitative stability results have also been established in the setting of Riemannian manifolds; see \cite{Ambrosio2019OptimalMap, kitagawa2025stability}. These stability estimates are closely related to the question of the time regularity of the optimal transport functional along curves in $\mathcal{P}_2$, a topic first investigated by Loeper \cite{loeper05stability} and subsequently developed by Gigli \cite{gigli}.
For a comprehensive overview of this rapidly developing field, we refer the reader to the Cours Peccot notes of Cyril Letrouit \cite{LetrouitPeccot}.
We also mention the very recent work \cite{chewistabilityKimMilman}, where an analogous quantitative stability result is established for the Kim–Milman transport map instead of the Brenier optimal transport map.

Our stability results (Theorems \ref{thm:stability} and \ref{thm:stabilitybis}) are not quantitative, as they rely on contradiction and compactness arguments. However, they hold in a stronger topology than those stated in the literature cited above, since they invoke the $\mathcal{C}^{3,\alpha}$ topology for the transport mapping instead of the $L^2$ norm, and the $\mathcal{C}^{2,\alpha}$ topology for the measure densities instead of the Wasserstein distance. We leave open the question of obtaining a quantitative bound in these topologies:
\[
\forall \mu_1,\mu_2\in \mathcal{P}_2,\quad \left\| T_{\rho\to\mu_1} - T_{\rho\to\mu_2}\right\|_{\mathcal{C}^{1,\alpha'}}  \leq C\,  \left\| f_{\mu_1} - f_{\mu_2}\right\|_{\mathcal{C}^{0,\alpha}},
\]
for some constants $C,\alpha>\alpha'>0$. Note, in particular, that such a bound may allow one to deduce how the $\varepsilon>0$ whose existence is guaranteed by Theorem \ref{thm:main} depends on the dimension.

\section{Proof of Theorem \ref{thm:main}}

Consider first the optimal transport map from $\vol_{g_{can}}$ onto $c\cdot\vol_{g_\rho}$, where $c=\vol_{g_{can}}(\Sph^n)/\vol_{g_{\rho}}(\Sph^n)$ stands for a normalizing constant and $g_\rho$ denotes the metric of the round sphere of radius $\rho$. Let us denote by $S : \Sph^n \to \Sph^n$ this transport map, which hence satisfies $S_\#\vol_{g_{can}} = c\cdot\vol_{g_\rho}$. It is easy to see that $S$ is equals to the composition $(\Sph^n, g_{can}) \hookrightarrow \R^{n+1} \to \R^{n+1}$, where $\hookrightarrow$ stands for the canonical embedding, and $\R^{n+1} \to \R^{n+1}$ is the dilation $x\mapsto \rho\,x$. As a consequence, by abuse of notation, we will often write that $S = \rho\,Id$.
On the other hand, we consider the optimal transport map from $\vol_{g_{can}}$ onto $c'\cdot\vol_{g}$, where $c'=\vol_{g_{can}}(\Sph^n)/\vol_{g}(\Sph^n)$ stands for a normalizing constant. Let us denote by $T : \Sph^n \to \Sph^n$ this transport map, which hence satisfies $T_\#\vol_{g_{can}} = c\cdot\vol_{g}$. 
Now, for a fixed $\delta>0$, by using Theorem \ref{thm:stability}, we get that if $\varepsilon$ is small enough in Condition \eqref{eq:nearlyspherical}, then it holds
\begin{equation}\label{eq:closenessOT}
\left|\left|S - T \right|\right|_{\C^{1,\alpha'}(\Sph^n)} \leq \delta
\end{equation}
So we can write
\begin{align*}
\left|\left| \nabla T \right|\right|_{g_{can}\to g} &\leq \left|\left| \nabla S - \nabla T \right|\right|_{g_{can}\to g} + \left|\left| \nabla S \right|\right|_{g_{can}\to g}
\end{align*}
where we denote by $\left|\left| \nabla T \right|\right|_{g_{can}\to g}$ the supremum norm of the Hilbert-Schmidt norm of the Jacobian of $S$, seen as a function between the following Riemannian manifolds $T : (\Sph^n, g_{can}) \to (\Sph^n, g) $.
We have that
\[
\left|\left| \nabla S - \nabla T \right|\right|_{g_{can}\to g}^2  = \sup_{x\in\Sph^n} \, g_{can}^{ij} g_{kl} \langle\nabla_i\left(S^k-T^k\right),  \nabla_j\left(S^l-T^l\right)\rangle
\]
where we use Einstein's summation convention. Therefore,
\[
\left|\left| \nabla S - \nabla T \right|\right|_{g_{can}\to g}^2  \leq \delta^2 \sup_{x\in\Sph^n} \, \sum_{i,j=1}^n g_{can}^{ii} \left[ (g_\rho)_{jj} + \varepsilon\right]   \leq K_\rho \delta^2
\]
for some constant $K_\rho$ only depending on $\rho$ and the dimension $n$. Moreover, since $S=\rho\,I_d$, it is clear that $\left|\left| \nabla S \right|\right|_{g_{can}\to g} \leq \rho $. As a consequence, it follows that
\[
\left|\left| \nabla T \right|\right|_{g_{can}\to g}   \leq \delta \sqrt{K_\rho} + \rho.
\]
And since $\rho<1$, it follows that we can choose $\delta$ such that $\delta \sqrt{K_\rho} < 1- \rho$, giving that the transport mapping $T$ is $1$-Lipschitz, which concludes the proof. $\qed$

\section{Proof of the stability Theorems \ref{thm:stability} and \ref{thm:stabilitybis}}

For ease of reading, we cite a few technical lemmas.
The first concerns the uniqueness result of the optimal map. 

\begin{lemma}{\cite[Theorem 8]{McCann01}} \label{LemmaMcCann} Let $(M,g)$ be a connected, compact Riemannian manifold, $C^3$ smooth and without boundary. Given two Borel probability measures $\mu,\nu $ on $(M, g)$ such that $\mu$ has a positive density function with respect to the volume element on $M$. Then there is a \emph{unique} optimal transport map pushing $\mu$ forwards to $\nu$ with respect to the quadratic cost.
\end{lemma}

If we consider the standard unit sphere $\Sph^n$ with the canonical metric, we have the following away from cut-locus estimate.
\begin{lemma}{\cite[Theorem 2]{delanoe2006gradient}} \label{Awaycut}
Let $\mu$ and $\nu$ be two Borel probability measures with positive densities with respect to the standard volume element: $\mu=\rho_1 dvol_{can}$ and $ \nu =\rho_2\, dvol_{can}$ on $\Sph^n$. 
Then the optimal transport map $T$ pushing forwards $\mu$ onto $\nu$ satisfies
$$
d(m, T(m))\le \pi -\frac{1}{2\pi}\left\{\frac{\inf\rho_1}{\sup \rho_2}\left[\frac{n\, \vol(\Sph^n)}{2\, \vol(\Sph^{n-1})}\right]^2\right\}^{1/n}
$$
\end{lemma}
\begin{proof}
In \cite[Theorem 2]{delanoe2006gradient}, we give the proof for $\rho_1\equiv 1$. But we could adapt their proof in the general case. We leave the detail for the interested readers.
\end{proof}

Still on the standard sphere $\Sph^n$ with the canonical round metric, we have the following regularity estimate.
\begin{lemma}{(\cite[Theorem 2.4]{loeper11} and \cite[Theorems 1.1 and 1.2]{liu2009interior})}\label{Regularity}
Let $\mu$ and $\nu$ be two Borel probability measures with positive densities with respect to the standard volume element $\mu=\rho_1\, dvol_{can}$ and $\nu =\rho_2\, dvol_{can}$ on $\Sph^n$. Assume that $$0<c_1<\rho_i<c_2$$ for $i=1,2$, where $c_1,c_2$ are positive constants. 
Let $T$ be the optimal transport map pushing forward $\mu$ onto $\nu$, and let $u : \Sph^n \to \R$ be its potential function satisfying \eqref{AMG}. Then for any $\beta\in (0,1)$, there exists some constant $C_\beta>0$ such that
\begin{equation}
\label{C1}
\|u\|_{C^{1,\beta}}\le C_\beta 
\end{equation}
Moreover, if we assume in addition that $$\rho_i\in C^{0,\alpha}$$ for $i=1,2$, then there exists some positive constant $C>0$ such that 
\begin{equation}
\label{C2}
\|u\|_{C^{2,\alpha}}\le C
\end{equation}
\end{lemma}
\begin{proof}
Applying Lemma~\ref{Awaycut}, we know that $T$ stays uniformly away from the cut locus:  
there exists $\delta>0$ such that, for every $x\in \Sph^n$,
\[
d(x,T(x)) \le \pi - \delta.
\]
In particular, the Ma--Trudinger--Wang (MTW) tensor of the cost is uniformly positive below and above on the graph of~$T$.  
Hence, by \cite[Theorem~2.4]{loeper11}, property \eqref{C1} follows.

Moreover, since $T$ is uniformly away from the cut locus and both densities satisfy  
$\rho_i \in C^{0,\alpha}$ for $i=1,2$, we deduce that
\[
\frac{f_1(x)}{\,f_2(T(x))\,}\, \bigl|\det \nabla_x \nabla_y c(x,T(x))\bigr|
    \in C^{0,\alpha}(\Sph^n).
\]
Therefore, by \cite[Theorems~1.1 and~1.2]{liu2009interior}, property \eqref{C2} holds as well.  
\end{proof}

We can now turn to the proof of Theorem~\ref{thm:stability}.  
\\

{\it Proof of Theorem \ref{thm:stability}}\\

The proof proceeds by a contradiction argument.
Let us then assume that for some $\delta_0>0$ and for any fixed $\alpha'\in (0,\alpha)$  one can find  a sequence $(f_1^k)_k\in \C^{0,\alpha}(\Sph^n,\R_+)$ such that
$$\forall k\in \N,\quad  \int_{\Sph^n} f_1^k\,d\vol_{can} = \vol_{can}(\Sph^n),$$
\begin{equation}\label{eq:cvfct}
\left|\left| f_1^k - 1\right|\right|_{\C^{0,\alpha}} \underset{k\to\infty}{\longrightarrow} 0
\end{equation}
and
\begin{equation}\label{eq:contrad}
\forall k\in \N,\quad \left|\left| T_1^k - Id \right|\right|_{\C^{1,\alpha'}} \geq \delta_0
\end{equation}
where $T_1^k$ stands for the optimal transport map sending $\vol_{can} $ onto $f_1^k\,d\vol_{can}$. Since the sequence $(f_1^k)_k$ converges in $\C^{0,\alpha}$, it follows that it is uniformly bounded in $\C^{0,\alpha}$, i.e.
$$\exists C>0,\quad \forall k\in \N,\quad \left|\left| f_1^k \right|\right|_{\C^{0,\alpha}}\leq C. $$
Therefore, by Lemma \ref{Regularity},  an a priori estimate holds, that is, 
\begin{equation}
\label{estimate1}
\exists C>0,\quad \forall k\in \N,\quad \left|\left| T_1^k \right|\right|_{\C^{1,\alpha}}\leq C.
\end{equation} 
In other words, the sequence of optimal transport maps $(T_1^k)_k$ is uniformly bounded in $\C^{1,\alpha}$. So by using the compact embedding of Hölder spaces (Arzela-Ascoli theorem)
$$\forall \alpha' \in (0,\alpha),\quad \C^{1,\alpha}(\Sph^n,\Sph^n) \hookrightarrow \C^{1,\alpha'}(\Sph^n,\Sph^n), $$ 
with the choice $ \alpha' \in (0,\alpha)$ given by the contradiction assumption \eqref{eq:contrad}, we can extract a subsequence, still denoted $(T_1^k)_k\in \C^{1, \alpha'}(\Sph^n,\Sph^n)$, which converges towards   some $S\in \C^{1, \alpha'}(\Sph^n,\Sph^n)$:
\begin{equation}\label{eq:convsubseq}
\left|\left| T_1^k - S \right|\right|_{\C^{1,\alpha'}} \underset{k\to\infty}{\longrightarrow} 0
\end{equation}
We are now going to prove that this limit $S$ is equals to the optimal transport map $Id$ between $\vol_{can}$ and $\vol_{can}$. First, we show that $S:\Sph^n\to \Sph^n$ is actually a transport map pushing forward $\vol_{can}$ onto $\vol_{can}$. Indeed, since for all $k$, $T_1^k$ is a transport map from $\vol_{can}$ onto $f_1^k\,d \vol_{can}$, it holds that for all continuous functions $\phi : \Sph^n\to\R$, and for all $k\in\N$,
\[
\int_{\Sph^n} \phi\circ T_1^k(y)\,d\vol_{can}(y) = \int_{\Sph^n} \phi(x)f_1^k(x)\,d\vol_{can}(x).
\]
Let $k\to\infty$, the convergence \eqref{eq:convsubseq} gives that the left-hand side converges to $\int_{\Sph^n} \phi\circ S(y)\,d\vol_{can}(y)$, and the convergence \eqref{eq:cvfct} gives that the right-hand side converges to $\int_{\Sph^n} \phi(x)\,d\vol_{can}(y)$, so we get that for all continuous functions $\phi : \Sph^n\to\R$,
\[
\int_{\Sph^n} \phi\circ S(y)\,d\vol_{can}(y) = \int_{\Sph^n} \phi(x)\,d\vol_{can}(y)
\]
whence $S_\#(\vol_{can}) = \,\vol_{can}$.
Second, arguing by continuity of the quadratic cost induced by the canonical distance on $\Sph^n$, we note that $S$ is optimal for the quadratic cost, as a limit of the optimal transport maps $T_1^k$ for this cost. Or equivalently, the corresponding potential function $u_k$ converges to $u$  the potential function of $S$, which satisfies the Monge-Amp\`ere equation (\ref{AMG}) for $f_1=f_2=1$. By \cite[Theorem 3.1]{delanoe2015smoothness}, $u$ is $c$-convex and thus, it is a potential function of $S$, which is an optimal transport map.

We have therefore seen that the map $S$ is an optimal transport map for the quadratic cost between the distributions $\vol_{can}$ and $\,\vol_{can}$. By uniqueness of this optimal map in Lemma \ref{LemmaMcCann}, we get that $$S = Id.$$
However, passing to the limit when $k\to\infty$ in \eqref{eq:contrad}, we get that
\[
\left|\left| S - Id \right|\right|_{\C^{1,\alpha'}} \geq \delta_0,
\]
which gives us the sought contradiction. The proof is achieved.\qed\\

We can now turn to the proof of Theorem~\ref{thm:stabilitybis}. To keep the exposition light and avoid unnecessary technical detours, we only provide a brief outline of the argument.\\

\textit{Sketch of the proof of Theorem~\ref{thm:stabilitybis}.}\\

By \cite[Lemma 4.1]{delanoe2015smoothness}, the linearized operator of 
\[
\displaystyle\frac{f_2(T(x))}{\left|\det \nabla_x\nabla_y\, c(x, T(x))\right|}\left(\nabla^{2}u + \nabla^{2}c(x,T(x))\right)
\]
at any admissible potential \(u_{1}\in C^{k+2,\alpha}\) is an isomorphism from \(C^{k+2,\alpha}_0\) onto \(C^{k,\alpha}_0\), for every \(k\ge 0\), where \(C^{k,\alpha}_0\) denotes the affine Banach space of
\(C^{k,\alpha}\) real functions with average on $M$ equal to $0$.  
Once this invertibility is known, the continuity method (see \cite[Section 4]{delanoe2015smoothness}) applies in a standard way and, combined with the inverse function theorem, yields the required stability estimate. We do not reproduce the full argument here, as it follows closely the reasoning in the cited reference.

Alternatively, when $k=0$, the proof is as same as the one of  Theorem \ref{thm:stability}. When $k>0$, we differentiate the equation (\ref{C2}) and the desired results follow from the classical regularity theory of linear elliptic PDE. We leave the detail for the interested readers.

\qed

\end{document}